\documentclass[a4paper,11pt,reqno,noindent]{amsart}
\usepackage[centertags]{amsmath}
\usepackage{amsfonts,amssymb,amsthm} 
\usepackage{hyperref}

 \hypersetup{
     pdfmenubar=false,        
     pdfnewwindow=true,      
     colorlinks=false,       
     linkcolor=blue,          
     citecolor=blue,        
     filecolor=magenta,      
     urlcolor=cyan           
 }
\usepackage{graphicx} 

\usepackage[english]{babel}
\usepackage{newlfont}
\usepackage{color}
\usepackage[body={15cm,21.5cm},centering]{geometry} 
\usepackage{fancyhdr}
\usepackage{esint}
\usepackage{enumerate}
\usepackage{cite}

\usepackage[active]{srcltx}

\newtheorem{theorem}{Theorem}
\newtheorem{lemma}{Lemma}

\newtheorem{corollary}{Corollary}
\newtheorem{rem}{Remark}

\theoremstyle{definition}
\newtheorem{defi}{Definition}

\newcommand{\supp}{\operatorname{supp}}

\newcommand{\e}{\varepsilon}

\newcommand{\R}{\mathbb{R}}
\newcommand{\Z}{\mathbb{Z}}
\newcommand{\N}{\mathbb{N}}

\renewcommand{\d}{\mathrm{d}}
\newcommand{\ddta}{\frac{\d}{\d t}}
\newcommand{\ddt}[1]{\frac{\d #1}{\d t}}
\newcommand{\ddtt}[2]{\frac{\d^{#1}#2}{\d t^{#1}}}
\newcommand{\ddtta}[1]{\frac{\d^{#1}}{\d t^{#1}}}
\renewcommand{\L}{\mathcal{L}}
\newcommand{\A}{\mathcal{A}}

\newcommand{\id}{\mathrm{Id}}

\newcommand\wto{\rightharpoonup}
\newcommand\wsto{\stackrel{*}{\rightharpoonup}}

\date{\today}
\author[H. Olbermann] {Heiner Olbermann}
\date{\today}
\address[Heiner Olbermann]{Hausdorff Center for Mathematics, Bonn, Germany}
\email{heiner.olbermann@hcm.uni-bonn.de}
\title{The one-dimensional model for d-cones revisited}

\begin{document}
\maketitle
\begin{abstract}
A \emph{d-cone} is the shape one obtains when pushing an elastic sheet at its
center into a hollow cylinder. In a simple model, one can treat the 
elastic sheet in the deformed configuration as a developable surface with a singularity at the ``tip'' of the
cone.
In this approximation, the renormalized elastic energy is 
given by the bending energy density integrated over some annulus in the
reference configuration. The thus defined  variational problem depends on the
indentation $h$ of the
sheet into the cylinder.
This model has been
investigated before in the physics
literature; the main motivation
for the present paper is to give a rigorous version of some of the results
achieved there via formal arguments.
We derive the Gamma-limit of the energy functional as $h$ is sent to
0. Further, we analyze the minimizers of the limiting functional, and list a
number of necessary conditions that they have to fulfill. 
\end{abstract}
\section{Introduction}
Since the late '90s, there has been a lot of interest in the crumpling of thin
elastic sheets  in the physics community
\cite{PhysRevE.71.016612,PhysRevLett.80.2358,Cerda08032005,RevModPhys.79.643,CCMM,PhysRevLett.87.206105,PhysRevLett.78.1303,Lobkovsky01121995,MR2023444}. These
works mainly treat what may be thought of as  ``building blocks'' of the
more complex folding patterns one obtains when crushing an elastic sheet into a
container whose size is smaller than the diameter of the sheet. In other words,
these contributions analyze single ridges or single vertices, where the elastic energy
focuses. In the mathematics literature, ridges have been investigated in
\cite{MR2358334}, and vertices in \cite{MR3102597,MR3168627}. More precisely, these latter
works 
considered the so-called d-cone, the shape that one
obtains when pushing a thin elastic sheet at its center into a hollow
cylinder, which in the physics literature has been treated in \cite{PhysRevE.71.016612,PhysRevLett.80.2358,CCMM,Cerda08032005,RevModPhys.79.643}. 
This is the physical setup we will be interested in here.\\
\\
In \cite{MR3102597,MR3168627}, the d-cone has been modeled by the following
variational problem. Let $\gamma\in W^{2,2}(S^1,S^2)$ be a unit speed curve that
is not contained in a plane,
denote the unit ball in $\R^2$ by $B$, identify its boundary with $S^1$, and set
\[
\mathcal Y=\left\{y\in W^{2,2}(B;\R^3):y|_{\partial
    B}=\gamma,\,y(0)=0\right\}\,.
\]
The elastic energy of $y\in\mathcal Y$ is given by
$I_H(y)=\int_B|Dy^TDy-\id_{2\times 2}|^2+H^2|D^2y|^2\d x$, where $H$ is a
parameter that can be thought of as the thickness of the sheet. (This is a
typical model energy for thin elastic sheets; for a justification see e.g.~\cite{MR2358334}.) The result of
\cite{MR3102597,MR3168627} is that $\inf_{y\in \mathcal Y}I_H(y)$ is equal to
$C(\gamma)H^2|\log H|$
as $H\to 0$ in the leading order of $H$, with an explicit
constant $C(\gamma)$.\\
\\
In \cite{PhysRevLett.80.2358,Cerda08032005},  the d-cone has been modeled  as a developable surface
with a singularity at its tip. 
The connection to \cite{MR3102597,MR3168627} is that the shape of  the d-cone
here is entirely determined  by the boundary curve $\gamma$ from above. The
energy in the present model is (up to numerical constants) given by
$C(\gamma)$, and we look for configurations $\gamma$ with minimal energy. 
This model is one-dimensional, and can
be treated with ODE methods. In \cite{PhysRevLett.80.2358,Cerda08032005}, quantitative results are given for the
regime of ``small deflections'', in which  nonlinear terms are dropped. The
resulting equation is a one-dimensional obstacle problem with an additional
constraint. It is argued that solutions of this problem should consist of a
finite number of ``folds'', i.e., regions where the sheet lifts off the edge of
the cylinder. The elastic energy of such configurations is computed numerically,
and the numerical evidence clearly suggests that the solution consisting of a
single fold (without any ``sub-folds'') is the configuration of lowest
energy. Since the small deflection regime is independent of the indentation $h$,
the conclusion is that the shape of this minimizer is universal. This means in
particular that the angle subtended by the region where the sheet lifts off the
cylinder is independent of the indentation or any other parameter such as
elastic moduli of the sheet or the radius of the cylinder. The value of this
angle is roughly 140$^\circ$, in good agreement with experimental
observations.\\
\\
Here, we give a rigorous derivation of the small deflection regime in the sense
of $\Gamma$-convergence. Additionally, we reconsider the limiting functional and
give a list of properties that have to be satisfied by its minimizers. This
second part is quite similar to the analysis in \cite{PhysRevLett.80.2358}.
However, we carefully derive the necessary conditions for minimizers with
purely variational tools, and 
our results are slightly different, in that  the necessary
conditions we find  are not quite strong enough to exclude a
certain set of configurations that has been missed in
\cite{PhysRevLett.80.2358,Cerda08032005}.\\
\\
The plan of the present paper is as follows: In Section
\ref{sec:deriv-limit-probl}, we define our model and state the
$\Gamma$-convergence result. We also give the proofs of the ``compactness'' and
``lower bound'' parts of this statement, which are straightforward. In Section \ref{sec:constr-recov-sequ}, we
prove the ``upper bound'' part, which is somewhat more complicated. The main
difficulty is to make sure the various constraints are satisfied by the recovery
sequence. In Section \ref{sec:minim-limit-funct}, we state
and prove a number of necessary conditions for minimizers of the limiting
functional. The proof relies on a generalized Lagrange multiplier rule from
\cite{dmitruk1980lyusternik}, valid for variations in a convex cone.\\
\\
{\bf Notation.} Let $S^1=\R/ (2\pi\Z)$ and let $\iota:\R\to S^1$ be the quotient map. When we
write $(a,b)\subset S^1$  for $a,b\in\R$, it is understood that we are speaking of
the image of the interval under the quotient map $\iota$. The function space
$W^{k,p}(S^1)$ is given by
\[
\begin{split}
  \Big\{f:S^1\to\R:&\exists \tilde f\in W^{k,p}_{\mathrm{loc}}(\R), \tilde
    f(x)=\tilde f(x+2\pi)=f(\iota(x))\text{ for all }x\in\R\Big\}\,.
\end{split}
\]
The spaces $C^k(S^1)$ are defined analogously. Letting $I=[-\pi,\pi)$ and using the above identification of
$S^1=\iota(I)$ with $I$, we define the $W^{k,p}$ norm on $S^1$ by
\[
\|f\|_{W^{k,p}(S^1)}=\|f\|_{W^{k,p}(I)}\,.
\]
For the derivative of a function
$f\in W^{1,1}(S^1)$, we use both the notation $f'$ and $\ddt{f}$. For the use of the symbols $C$ and
$O$ in Section
\ref{sec:constr-recov-sequ}, see the explanations at the beginning of that section.\\
\\
{\bf Acknowledgments.} The author would like to thank Stefan M\"uller for
suggesting this problem.

\section{Derivation of the small deflection regime by $\Gamma$-convergence}
\label{sec:deriv-limit-probl}
The starting point is the variational problem given by the elastic energy
\[
\begin{array}{rrl}
E_{\text{bending}}:&W^{2,2}(S^1;\R^3)&\to \R\\
 & \gamma &\mapsto\begin{cases}\int_{0}^{2\pi}|\gamma''+\gamma|^2 \d t & \text{
     if } |\gamma|=|\gamma'|=1 \text{ a.e.}\\+\infty & \text{ else.}\end{cases}
\end{array}
\]
This is (up to a constant) the bending energy  $\int |\nabla^2 y|^2$ of an
elastic sheet $\Omega=B(0,1)\setminus B(0,\e)\subset\R^2$ under the deformation
$y:\Omega \to \R^3$, $y(r,t)=r \gamma(t)$ (the latter of course in
polar coordinates). The constraints $|\gamma|=|\gamma'|=1$ assure that $y$ is
an isometry away from the origin. The energy $E_{\text{bending}}$ is the leading
term in the energy scaling result for d-cones with boundary conditions given by
$\gamma$, see \cite{MR3102597,MR3168627}.\\
\\
Now we define the constrained functional
\newcommand{\E}{{\mathcal E}}
\[
\begin{array}{rrl}
\E_h:&W^{2,2}(S^1;\R^3)&\to \R\\
 & \gamma &\mapsto\begin{cases}E_{\text{bending}}(\gamma) & \text{ if }\gamma\cdot e_z\geq h\\\infty & \text{ else. }\end{cases}
\end{array}
\]
This models the d-cone being pushed into a cylinder of  height $h$ and radius $\sqrt{1-h^2}$.
The limit functional for $h\to 0$ will also be defined on the space $W^{2,2}(S^1;\R^3)$. There will be constraints for allowed configurations, and we define the space of admissible deformations:
\[
\mathcal A=\left\{(u,v,w)\in W^{1,\infty}(S^1;\R^2)\times W^{2,2}(S^1): w\geq 1, \, u=-w^2/2,\,u+v'=-w'^2/2\right\}\,.
\]
The constraints captured in the definition of $\A$ are  the remnants of the
constraints $\gamma_h\cdot e_z\geq h$, $|\gamma_h|=|\gamma_h'|=1$ for finite $h$.
We define the limit functional
\[
\begin{array}{rrl}
\E_0:&W^{1,\infty}(S^1;\R^2)\times W^{2,2}(S^1)&\to \R\\
 & (u,v,w) &\mapsto\begin{cases}\int_0^{2\pi}|w''+w|^2\d t & \text{ if }(u,v,w)\in\A\\\infty & \text{ else. }\end{cases}
\end{array}
\]
In the following, we write 
\begin{equation}
\begin{split}
  u_h(t)=&\gamma_h(t)\cdot e_r(t)-1\\
  v_h(t)=&\gamma_h(t)\cdot e_\varphi(t)\\
  w_h(t)=&\gamma_h(t)\cdot e_z
\end{split}\label{eq:20}
\end{equation}
for $t\in S^1$, where we have introduced the $t$-dependent orthonormal frame
\[
e_r(t)=(\cos t,\sin t,0),\quad e_\varphi(t)=(-\sin t,\cos t,0),\quad e_z=(0,0,1)\,.
\]
We will prove the following $\Gamma$-convergence result:
\begin{theorem}
\label{thm:gamma}
\emph{Compactness:} If $\gamma_h$ is a sequence in $W^{2,2}(S^1;\R^3)$ with
$\limsup_{h\to 0} h^{-2}\E^h(\gamma_h)<\infty$, then there exists a subsequence
(no relabeling)
and $(u,v,w)\in \A$ such that
\[
\begin{split}
h^{-1}w_h\wto& w \text{ in }W^{2,2}(S^1)\\
h^{-2}(u_h,v_h)\wsto & (u,v)\text{ in }W^{1,\infty}(S^1;\R^2) 
\end{split}
\]
\emph{Lower bound:} Let $\gamma_h$ be a sequence in $W^{2,2}(S^1;\R^3)$ such
that for $u_h,v_h,w_h$ defined as in \eqref{eq:20}, we have $h^{-1}w_h\wto w$ in
$W^{2,2}(S^1)$ and
$h^{-2}(u_h,v_h)\wsto  (u,v)$ in $W^{1,\infty}(S^1;\R^2)$. Then
\[
\liminf_{h\to 0}h^{-2}\E_h(\gamma_h)\geq \E^0(u,v,w)\,.
\]
\emph{Upper bound:} Let $(u,v,w)\in W^{1,\infty}(S^1;\R^2)\times
W^{2,2}(S^1)$. Then there exists a sequence $\gamma_h$ in $W^{2,2}(S^1;\R^3)$
such that 
$h^{-1}w_h\wto w$ in
$W^{2,2}(S^1)$ and
$h^{-2}(u_h,v_h)\wsto  (u,v)$ in $W^{1,\infty}(S^1;\R^2)$, and additionally
\[
\lim_{h\to 0}h^{-2}\E_h(\gamma_h)= \E^0(u,v,w)\,.
\]
\end{theorem}

\begin{proof}[Proof of compactness and lower bound.]
Using the notation from \eqref{eq:20}, we have
\begin{equation}
\E_h(\gamma_h)=\int \left(2 v_h'-u_{h}''\right)^2+\left(2
  u_{h}'+v_{h}''\right)^2+\left(w_h''+w_h\right)^2\,.\label{eq:1}
\end{equation}
By  the coercivity of $\E_h$ in $W^{2,2}(S^1;\R^3)$ and $\limsup_{h\to 0} h^{-2}
\E^h(\gamma_h)<\infty$, 
\begin{equation}
h^{-1}(u_h,v_h,w_h) \text{ is bounded in }W^{2,2}(S^1;\R^3)\label{eq:22}\,,
\end{equation} 
and hence a subsequence converges to some $(U,V,w)\in
W^{2,2}(S^1;\R^3)$. By $h^{-1}w_h\geq 1$, we have $w\geq 1$. By the constraints $|\gamma_h|=1$, $|\gamma_h'|=1$, 
\begin{subequations}
  \begin{align}
    |\gamma_h|^2=(1+u_{h})^2+v_{h}^2+w_h^2=&1 \label{eq:21}\\
    |\gamma_h'|^2=(1+u_{h}+v_{h}')^2+(u_{h}'-v_{h})^2+w_h'^2=&1\,.
    \label{eq:14}
  \end{align}
\end{subequations}
Differentiating \eqref{eq:21} twice, we get 
\[
  -2u_{h}''=2u_{h}u_h''+u_h'^2+2v_{h}v_h''+v_h'^2+2w_{h}w_h''+w_h'^2\,.
\]
Multiplying this equality with $h^{-2}$, and using \eqref{eq:22} and H\"older's inequality
on the right hand side, we get boundedness of $h^{-2}u_h$ in $W^{2,1}(S^1)$ and
hence in $W^{1,\infty}(S_1)$. In the same way, by differentiating \eqref{eq:14}
twice, we get
\[
-2(u_h'+v_h'')= 2(u_h'+v_h'')(u_h+v_h')+2(u_h''-v_h')(u_h'-v_h)+2w_h''w_h'\,.
\]
Again multiplying by $h^{-2}$, using \eqref{eq:22} and the fact that $h^{-2}u_h$
is bounded in $W^{1,\infty}(S^1)$, we get 
boundedness of $h^{-2}v_h$ in $W^{1,\infty}(S_1)$. Choose a convergent
subsequence such that $h^{-2}(u_h,v_h)\wsto (u,v)$ in $W^{1,\infty}(S_1)$.\\
Multiplying  \eqref{eq:21} and \eqref{eq:14} by $h^{-2}$ and taking the limit
$h\to 0$  (say, the weak-* limit in $L^\infty$), we get
\[
\begin{split}
  -2u=&w^2\\
  -2(u+v')=&w'^2\,.
\end{split}
\]
We conclude that $(u,v,w)\in \A$.
This proves the compactness part. The lower bound follows immediately from 
formula \eqref{eq:1} for $\E_h$ and the weak lower semi-continuity of 
\[
w\mapsto\int_{S^1} (w''+w)^2\d t
\]
in $W^{2,2}(S_1)$. 
\end{proof}
\section{Construction of the recovery sequence in Theorem \ref{thm:gamma}}
\label{sec:constr-recov-sequ}
We will construct a recovery sequence $\gamma_h\in W^{2,2}(S^1,S^2)$ that meets
the constraints $|\gamma_h|=|\gamma'_h|=1$, $\gamma\cdot
e_z\geq h$ in several
steps. We start off with some sequence $\gamma_h^{(1)}:S^1\to\R^3$, and each step
$\gamma_h^{(i)}\to \gamma_h^{(i+1)}$ shall assure that one additional constraint
is met. At first, we will give this sequence of modifications for $(u,v,w)\in
\A\cap W^{2,\infty}(S^1;\R^3)$. 
The proof will be completed by an approximation
argument. \\
In Lemma \ref{lem:defo14} and Lemma \ref{lem:assver} below, $(u,v,w)\in\A\cap W^{2,\infty}(S^1;\R^3)$ will
be fixed, and we will use the following notational convention: \\
A statement such as ``$f\leq Cg$'' will be shorthand for the statement ``There
exists a constant $C>0$ that only depends on $\|u\|_{W^{2,\infty}}$,
$\|v\|_{W^{2,\infty}}$ and $\|w\|_{W^{2,\infty}}$, such that $f\leq Cg$.''
Similarly, if $f$ and $g$ depend on $h$, we will write $f=g+O(h^k)$ if there exists a constant $C$ that only 
depends on $\|u\|_{W^{2,\infty}}$,
$\|v\|_{W^{2,\infty}}$ and $\|w\|_{W^{2,\infty}}$, such that $|f-g|\leq C h^k$
for all $h$.
\begin{lemma}
\label{lem:defo14}
Let $(u,v,w)\in\A\cap W^{2,\infty}(S^1;\R^3)$. Then there exists a sequence of curves
$\gamma_h^{(4)}:\R\supset[0,2\pi]\to S^2$ with the following properties for $h$
small enough:
\begin{align}
\left|\ddt{\gamma^{(4)}_h}\right|=&1\,,\label{eq:16}\\
\gamma^{(4)}_h\cdot e_z\geq& h+\frac12 h^{5/2}\,,\label{eq:15}\\
|\gamma_h^{(4)}(2\pi)-\gamma_h^{(4)}(0)|\leq & Ch^4\,,\label{eq:gestim}\\
\left|\ddt{\gamma_h^{(4)}}(2\pi)-\ddt{\gamma_h^{(4)}}(0)\right|\leq & Ch^4\,.\label{eq:Dgestim}
\end{align}
\end{lemma}
\begin{proof}
The initial ansatz is to define $\gamma_h^{(1)}:S^1\to\R^3$ by
\[
\gamma_h^{(1)}=(1+h^{2}u)e_r+h^{2}ve_\varphi+hwe_z\,.
\]
To make sure that the constraint $\gamma_h\cdot e_z\geq h$ holds even after the
modifications we are going to perform in the sequel, we set

\begin{equation}
\gamma_h^{(2)}=\gamma_h^{(1)}+h^{5/2}e_z\,.\label{eq:23}
\end{equation}
By a computation using $(u,v,w)\in \A$, we have
\begin{subequations}
\begin{align}
  |\gamma_h^{(2)}|^2=&1+h^4(u^2+v^2+2w)+h^5\label{eq:24}\\
  \left|\ddt{\gamma_h^{(2)}}\right|^2=&1+h^4\left((u+v')^2+(u'-v)^2\right)\,.\label{eq:25}
\end{align}
\end{subequations}
For future reference, we also make the following computations, 
\begin{equation}
\begin{split}
\left|\ddtt{2}{\gamma_h^{(2)}}\right|=&\left|(-1+h^2(u''-u-v'))e_r+h^2(u'+v''-v)e_\varphi+hw''e_z\right|\\
\leq& 1+Ch\\
  \ddta\left(|\gamma_h^{(2)}|^2\right)=&\ddta\left(1+h^2\underbrace{(2u+w^2)}_{=0}+h^4(u^2+v^2+2w)+h^5\right)\\
  =& 2h^4\left(uu'+vv'+w'\right)\\
\leq & Ch^4\\
    \frac{\d^2}{\d t^2}\left(|\gamma_h^{(2)}|^2\right)\leq &Ch^4
  \end{split}\label{eq:6}
\end{equation}
Our next modification assures the constraint $|\gamma_h|=1$. Namely, we define
$\gamma_h^{(3)}:S^1\to \R^3$ by
\[
\gamma_h^{(3)}=\gamma_h^{(2)}/|\gamma_h^{(2)}|\,.
\]
Note that $\gamma_h^{(3)}\cdot e_z\geq h+frac12 h^{5/2}$ by \eqref{eq:23} and
\eqref{eq:24} for $h$ small enough.
However, $\gamma_h^{(3)}$  does not fulfill the constraint $|\gamma_h'|=1$:
\begin{equation}
  \label{eq:10}
  \left|\ddt{\gamma_h^{(3)}}\right|=\frac{\ddta{\gamma_h^{(2)}}}{|\gamma_h^{(2)}|}-\frac12\frac{(\ddta|\gamma_h^{(2)}|^2)\gamma_h^{(2)}}{|\gamma_h^{(2)}|^3}=1+O(h^4)\,.
\end{equation}
Hence  we get for the length $ L_h$ of $\gamma_h^{(3)}$,
\begin{equation}
\begin{split}
  L_h=&\int_0^{2\pi}\left|\ddt{\gamma_h^{(3)}}\right|\d t\\
  =&2\pi+O(h^4)\,.
\end{split}\label{eq:26}
\end{equation}
The next step is a re-parametrization of $\gamma_h^{(3)}$, which will assure the
condition $|\gamma_h'|=1$, at the expense of the curve being closed. 
We
define $\tau_h:\R\supset[0,\infty)\to \R$ by
\[
  \tau_h(s)=\int_0^s\left|\ddt{\gamma_h^{(3)}}\right|^{-1}\d t\,.
\]
(Recall that by our notational convention, we do not distinguish on the right
hand side between $\gamma_h^{(3)}$ and $\gamma_h^{(3)}\circ\iota$, where $\iota:\R\to S^1$ is the canonical projection.)
Next we define $\gamma_h^{(4)}:\R\supset[0,\infty)\to S^2$ by
\[
  \gamma_h^{(4)}=\gamma_h^{(3)}\circ \tau_h\,.
\]
Note that $\gamma_h^{(4)}$ automatically satisfies
\eqref{eq:16}. Moreover, \eqref{eq:15} is satisfied since $\gamma^{(3)}_h$
fulfilled that property too. 
Further, by $\gamma_h^{(4)}(L_h)=\gamma_h^{(4)}(0)$, $\ddt{\gamma_h^{(4)}}(L_h)=\ddt{\gamma_h^{(4)}}(0)$ and \eqref{eq:26},
\begin{subequations}
\begin{align}
  |\gamma^{(4)}_h(2\pi)-\gamma^{(4)}_h(0)|\leq &C h^4 \sup\left|\ddt{\gamma^{(4)}_h}\right|\label{eq:13}\\
  \left|\ddt{\gamma^{(4)}_h}(2\pi)-\ddt{\gamma^{(4)}_h}(0)\right|\leq &C h^4
  \sup\left|\ddtt{2}{\gamma_h^{(4)}}\right|\,.\label{eq:28}
\end{align}
\end{subequations}
We estimate the suprema on the right hand sides,
\begin{subequations}
\begin{align}
 \sup\left|\ddt{\gamma^{(4)}_h}\right|\leq &(\sup|\tau_h'|)\sup\left|\ddt{\gamma_h^{(3)}}\right|\nonumber\\
  \leq&\left(\sup \left|\ddt{\gamma_h^{(3)}}\right|^{-1}\right)\sup\left|\ddt{\gamma_h^{(2)}}\right|\nonumber\\
  =&1+O(h^4)\label{eq:12}\\
\sup\left|\ddtt{2}{\gamma^{(4)}_h}\right|\leq&
\sup|\tau_h''|\sup\left|\ddt{\gamma^{(3)}_h}\right|+\sup|\tau_h'|\sup\left|\ddtt{2}{\gamma^{(3)}_h}\right|\,.\label{eq:27}
\end{align}
\end{subequations}
The  estimate \eqref{eq:12} in combination with \eqref{eq:13} proves \eqref{eq:gestim}, and it remains to prove
\eqref{eq:Dgestim}. We first compute $\ddtt{2}{\gamma_h^{(3)}}$, using
the notation $f=|\gamma_h^{(2)}|^{-1}$, 
\begin{equation}
   \ddtt{2}{\gamma_h^{(3)}}=
   \frac{\d^2}{\d t^2}(f\gamma_h^{(2)})=\left(\ddtt{2}f\right)\gamma_h^{(2)}+2\left(\ddt f\right)\left(\ddt{\gamma_h^{(2)}}\right)+f\ddtt{2}{\gamma_h^{(2)}}\,.
\label{eq:5}
\end{equation}
The derivatives of $f$ are estimated as follows,
\[
\begin{split}
  \ddt f=&-\frac12 |\gamma_h^{(2)}|^{-3}{\ddta|\gamma_h^{(2)}|^2}\\
  \leq& Ch^4\\
 \ddtt{2}f= & -\frac12 |\gamma_h^{(2)}|^{-3} \ddtta{2}|\gamma_h^{(2)}|^2
  +\frac34|\gamma_h^{(2)}|^{-5}\left({\ddta|\gamma_h^{(2)}|^2}\right)^2\\
  \leq & Ch^4\,,
\end{split}
\]
where we have used \eqref{eq:6}. 
Inserting into \eqref{eq:5}, and using again \eqref{eq:6}, we get
\begin{equation}
\left|\ddtt{2}{\gamma_h^{(3)}}\right|\leq 1+Ch\,.\label{eq:9}
\end{equation}
Next we compute $\tau_h''$,
\begin{equation}
\begin{split}
  |\tau_h''|=&\frac12\left|\ddt{\gamma_h^{(3)}}\right|^{-3}\left|\ddta\left|\ddt{\gamma_h^{(3)}}\right|^2\right|\\
  \leq & \left|\ddt{\gamma_h^{(3)}}\right|^{-2} \left|\ddtt{2}{\gamma_h^{(3)}}\right|\\
\leq & 1+Ch\,,
\end{split}\label{eq:11}
\end{equation}
where we have used \eqref{eq:10} and \eqref{eq:9}. 
Inserting \eqref{eq:11} into
\eqref{eq:27}, we get 
\[
\sup\left|\ddtt{2}{\gamma_h^{(4)}}\right|\leq C\,.
\]
Thus by \eqref{eq:28}, we have proved \eqref{eq:Dgestim}. Finally, we reduce the
domain of $\gamma_h^{(4)}$ from $[0,\infty)$ to $[0,2\pi]$. This completes the
proof of the lemma.
\end{proof}

The final step in the construction of the recovery sequence is gluing the ends
of the non-closed curve  $\gamma_h^{(4)}$ back together. 
For the sake of brevity, let us write $\gamma_h^{(4)}=\bar\gamma$. Further, we
introduce an orthonormal-frame-valued map $F$ by
$U=\bar\gamma$, $T=\bar\gamma'$, $N=T\wedge U$ and $F=(T,N,U)^T$. Finally, let $\kappa=N\cdot
\bar\gamma''$. Then $F$ (and in particular, $\bar \gamma$) is determined by initial conditions and the ODE
\begin{equation}
F'=\left(\begin{array}{ccc} 0&\kappa&-1\\-\kappa &0&0\\1& 0& 0\end{array}\right)F\,.\label{eq:2}
\end{equation}
We will have to modify $F$ such that $F(0)=F(2\pi)$.\\
\\
We introduce the following notation for modifications of curves:
\begin{defi}
For curves  $\gamma\in W^{2,2}([0,2\pi]; S^2)$ with $|\gamma'|=1$, and
$\psi\in C_0^\infty((0,2\pi))$, we define
$\kappa=\gamma''\cdot(\gamma'\wedge\gamma)$ and let  $F_\gamma[\psi]=(T,N,U)^T$  be the unique solution of the
initial value problem
\[\left\{
\begin{split}
  F_\gamma[\psi](0)=&\left(\ddt\gamma(0),\ddt\gamma(0)\wedge\gamma(0),\gamma(0)\right)^T\\
  \ddta F_\gamma[\psi](t)=&\left(\begin{array}{ccc} 0&\kappa(t)+\psi(t)&-1\\-(\kappa(t)+\psi(t)) &0&0\\1& 0&
      0\end{array}\right)F_\gamma[\psi](t)\quad\text{ for all }t\in[0,2\pi]\,.\label{eq:2}
\end{split}\right.
\]
Furthermore, we set
 $\gamma[\psi](t)=U(t)$ and $\ddt{\gamma}[\psi](t)=T(t)$ for all
 $t\in[0,2\pi]$. (Of course, this definition satisfies $\ddta\left(\gamma[\psi](t)\right)=\ddt{\gamma}[\psi](t)$.)
\end{defi}

Another tool in the modification process will be the following standard
implicit function theorem (see e.g.~\cite{MR787404}):
\begin{theorem}
\label{thm:implicit}
Let $x_0\in\R^n$, $f:B_r(x_0)\to \R^n$
continuously differentiable. Further let $\alpha,\beta,k\in \R^+$ such that 
\[
\begin{split} {\text a) }&\,Df(x_0) \text{ is invertible, }\quad
  |Df(x_0)^{-1}f(x_0)|\leq \alpha\,,\quad
  |Df(x_0)^{-1}|\leq \beta\\
  {\text b) }&\,|Df(x_1)-Df(x_2)|\leq k |x_1-x_2|\quad \text{for all }x_1,x_2\in B_r(x_0)\\
  {\text c) }&\,2\alpha\beta k<1\,\text{ and } 2\alpha<r\,.
\end{split}
\]
Then there exists a unique
solution $y\in B_r(x_0)$ to $f(y)=0$.
\end{theorem}
For $t\in S^1$, we have $\bar\gamma(t)\in S^2$ and $\ddt{\bar\gamma}(t)\in S^2\cap
T_{\gamma(2\pi)}S^2\simeq S^1$, where $T_pS^2$ denotes the tangent space of
$S^2$ at $p$. I.e., the range of the map $t\mapsto
(\bar\gamma(t),\ddt{\bar\gamma}(t))$ is the bundle
\[
M=\{(x,T)\in S^2\times S^2: T\in T_xS^2\}\,.
\]
This is a 3-dimensional manifold.
In a small enough neighborhood $U$ of a point $(x,T)\in M$, a chart is given by
\[
\begin{split}
  \zeta:U\to &\R^3\\
  (\bar x,\bar T)\mapsto &\left(\bar x-x(x\cdot \bar x),(T\wedge x)\cdot \bar T\right)\,,
\end{split}
\]
Here, in the first two components, we made the identification
$\{y\in\R^3:y\cdot x=0\}\simeq \R^2$.
Below, we will choose the chart $\zeta$ defined as above with
$x=\bar\gamma(2\pi)$ and $T=\ddt{\bar\gamma}(2\pi)$.

Now, for $\bar \psi\in C_0^\infty((0,2\pi);\R^3)$, we set

\begin{equation}
 \label{fdef}
\begin{split}
  f_h:B(0,r)\subset \R^3\to &\R^3\\
  a\mapsto& \zeta\circ \left(\bar\gamma[a\cdot\bar\psi](2\pi),\ddt{\bar\gamma}[a\cdot\bar\psi](2\pi)\right)- \zeta\circ\left(\bar\gamma(0),\ddt{\bar\gamma}(0)\right)\,,
\end{split}
\end{equation}
 and we want to get existence of
$a\in B_r(0)$ such that $f_h(a)=0$.
The heart of the matter will be the application of Theorem
\ref{thm:implicit}. The upcoming lemma assures that its conditions are met
for  $\bar\gamma=\gamma_h^{(4)}$ as in the conclusion of Lemma \ref{lem:defo14}
for $h$ small enough.

\begin{lemma}
\label{lem:assver}
Let $u,v,w\in W^{2,\infty}(S^1)\cap\A$ and $\bar\gamma=\gamma_h^{(4)}$ as in the
conclusion of Lemma \ref{lem:defo14}.
 There exists $\bar \psi\in C^\infty_0((0,2\pi);\R^3)$ such that for every $h$ small enough,
 \begin{itemize}
  \item The derivative $Df_h$ of the function $f_h$  defined in \eqref{fdef} has full rank at $a=0$.
 \item There exist constants $\alpha,\beta>0$ that do not depend on $h$ such that
\begin{align}
   |Df_h(0)^{-1}|\leq& \beta h^{-1},\label{eq:3}\\
 |Df_h(0)^{-1}f_h(0)|\leq &\alpha h^3\label{eq:4}\,.
\end{align}
 \item There exist $r_0, k>0$ that do not depend on $h$ such that $f_h$ is $C^2$ on $B(0,r_0)$, and 
\begin{equation}
\label{eq:Lipfh}
\sup_{a\in B(0,r_0)}|D^2 f_h(a)|\leq k \,.
\end{equation}

 \end{itemize}

\end{lemma}

\begin{proof}

Let $\psi\in C_0^\infty(0,2\pi)$. 
In this proof, we will write $F_{\bar\gamma}=F$. We start with the computation of the 
 derivative of $\R\to W^{2,2}(S^1,S^2)$, $\e\mapsto
 F[\e\psi](2\pi)$.\\ 
Set 
\[
 E=\left(\begin{array}{ccc}0&1&0\\-1&0&0\\0&0&0\end{array}\right)\,.
\]

We have
\[
\begin{split}
 \frac{\d}{\d\e}F[\e\psi](t)|_{\e=0}
=&\frac{\d}{\d\e}\int_0^t \left(\begin{array}{ccc}
0 & \kappa+\e\psi & -1\\
-(\kappa+\e\psi) & 0 & 0\\
1 & 0 & 0
\end{array}\right)(s) F[\e\psi](s)\d s\\
=& \int_0^t \psi(s) E F(s)+\left(\begin{array}{ccc}
0 & \kappa & -1\\
-\kappa & 0 & 0\\
1 & 0 & 0
\end{array}\right)(s) \frac{\d}{\d\e}F[\e\psi](s)|_{\e=0}\d s\\
\end{split}
\]
By the variation of constants formula, we get
\[
\begin{split}
 \frac{\d}{\d\e}F[\e\psi](t)|_{\e=0}=& F(t)\int_0^t \psi(s)F^{-1}(s)EF(s)\d s\\
=& F(t)\int_0^t \psi(s)(-N,T,0)\left(\begin{array}{c}T^T\\N^T\\U^T\end{array}\right)\d s\\
=& F(t)\int_0^t \psi(s)
\left(\begin{array}{ccc}0 & -U_3 & U_2\\ U_3 & 0 & -U_1\\-U_2 & U_1 & 0\end{array}\right)\d s
\end{split}
\]
In particular, this yields
\begin{equation}
\begin{split}
 \frac{\d}{\d\e}\bar\gamma[\e\psi](2\pi)|_{\e=0}=& \int_0^{2\pi} \bar\gamma(2\pi)\wedge \bar\gamma(s) \psi(s)\d s\\
 \frac{\d}{\d\e}\left.\left(\ddt{\bar\gamma}[\e\psi](2\pi)\right)\right|_{\e=0}=& \int_0^{2\pi} \ddt{\bar\gamma}(2\pi)\wedge \bar\gamma(s) \psi(s)\d s
\end{split}\label{eq:36}
\end{equation}
Repeating these arguments, we can compute the second derivative of $(\e_1,\e_2)\mapsto\bar\gamma[\e_1\psi_1+\e_2\psi_2](2\pi)$,

\begin{equation*}
\begin{split}
 \frac{\d}{\d\e_2}\frac{\d}{\d\e_1}&\left(\bar\gamma [\e_1\psi_1+\e_2\psi_2](2\pi)\right)|_{\e_1,\e_2=0}\\
=&\int_0^{2\pi}\left(\int_0^{2\pi} \bar\gamma(2\pi)\wedge \bar\gamma(t)\psi_2(t)\d t\right)\wedge\bar\gamma(s)\psi_1(s)\d s\\
&+\int_0^{2\pi}\bar\gamma(2\pi)\wedge
\left(\int_0^s\d t \psi_2(t)\bar\gamma(s)\wedge \bar\gamma(t)\right)
\psi_1(s)\d s\\
=& \int_0^{2\pi}\int_0^{s}\psi_1(s)\psi_2(t)\bar\gamma(2\pi)\wedge
\underbrace{\left(\bar\gamma(t)\wedge\bar\gamma(s)+\bar\gamma(s)\wedge\bar\gamma(t)\right)}_{=0}\d
t\d s\\
&+\int_0^{2\pi}\left(\int_s^{2\pi} \bar\gamma(2\pi)\wedge \bar\gamma(t)\psi_2(t)\d t\right)\wedge\bar\gamma(s)\psi_1(s)\d s
\end{split}\label{eq:37}
\end{equation*}

Similarly, we get

\begin{equation*}
\begin{split}
 \frac{\d}{\d\e_2}\frac{\d}{\d\e_1}&\left.\left(\ddt{\bar\gamma} [\e_1\psi_1+\e_2\psi_2](2\pi)\right)\right|_{\e_1,\e_2=0}\\
=&\int_0^{2\pi}\left(\int_s^{2\pi} \ddt{\bar\gamma}(2\pi)\wedge \bar\gamma(\bar
  t)\psi_2(\bar t)\d \bar t\right)\wedge\bar\gamma(s)\psi_1(s)\d s
\end{split}\label{eq:38}
\end{equation*}

Now, for arbitrarily chosen $\bar\psi\in C^\infty_0((0,2\pi);\R^3)$ and
arbitrary $a\in \R^3$ let $\tilde \gamma:=\bar\gamma[a\cdot\bar \psi]$. We can
repeat the above computations to obtain

\begin{subequations}
\begin{align}
  \frac{\d}{\d\e_2}\frac{\d}{\d\e_1}&\left(\tilde\gamma
    [\e_1\psi_1+\e_2\psi_2](2\pi)\right)|_{\e_1,\e_2=0}\nonumber\\
  &=\int_0^{2\pi}\left(\int_s^{2\pi} \tilde\gamma(2\pi)\wedge \tilde\gamma(t)\psi_2(t)\d t\right)\wedge\tilde\gamma(s)\psi_1(s)\d s\label{eq:41}\\
  \frac{\d}{\d\e_2}\frac{\d}{\d\e_1}&\left.\left(\ddt{\tilde\gamma} [\e_1\psi_1+\e_2\psi_2](2\pi)\right)\right|_{\e_1,\e_2=0}\nonumber\\
  &=\int_0^{2\pi}\left(\int_s^{2\pi} \ddt{\tilde\gamma}(2\pi)\wedge \tilde\gamma(\bar
    t)\psi_2(\bar t)\d \bar t\right)\wedge\tilde\gamma(s)\psi_1(s)\d s \,.\label{eq:39}
\end{align}
\end{subequations}

It is easily seen from the definition of $f_h$ that
\[
\begin{split}
  \left|\partial_i\partial_j f_h(a)\right|^2\leq&
  \left(\frac{\d}{\d\e_2}\frac{\d}{\d\e_1}\left(\tilde\gamma [\e_1\bar
      \psi_i+\e_2\bar \psi_j](2\pi)\right)|_{\e_1,\e_2=0}\right)^2\\
 & +\left(\frac{\d}{\d\e_2}\frac{\d}{\d\e_1}\left.\left(\ddt{\tilde\gamma}
      [\e_1\bar\psi_i+\e_2\bar\psi_j](2\pi)\right)\right|_{\e_1,\e_2=0}\right)^2\,.
\end{split}
\]
Hence, \eqref{eq:Lipfh} follows from the observation that the right hand sides
in \eqref{eq:41} and \eqref{eq:39} are bounded by a constant $\bar k$ that only
depends on $\psi_1=\bar\psi_i$ and $\psi_2=\bar\psi_j$, for any choice of $\bar\psi$.\\
Next we want to compute the determinant of $Df_h(0)$. Denoting by 
$M_{ij}$ the 2 by 2 minor of $Df_h$ that is obtained by deleting the $i$th
row and the $j$th column, we have
\begin{equation}
\det Df_h= \sum_{i=1}^3(-1)^{i+1}M_{3i}\partial_i(f_h)_3 \,.\label{eq:42}
\end{equation}
We recall that the first two components of $f_h(a)$ are the projection of
$\bar\gamma[a\cdot\bar\psi](2\pi)-\bar\gamma(0)$ to
$T_{\bar\gamma(2\pi)}S^2\simeq \R^2$. 
Assume $\{i,j,k\}=\{1,2,3\}$.
By the remark we just made, the minor $M_{3i}$ is (up to a sign) the
determinant of the 2 by 2 matrix formed by the partial derivatives of
$\bar\gamma[a\cdot\bar\psi](2\pi)$ with respect to $a_j$ and $a_k$ in some
orthonormal basis of $T_{\bar\gamma(2\pi)}S^2$. In other words,
\[
M_{3i}=\epsilon_{jk}\bar\gamma(2\pi)\cdot
\left(\frac{\d}{\d\e}\bar\gamma[\e\bar\psi_j](2\pi)|_{\e=0}\right)\wedge
\left(\frac{\d}{\d\e}\bar\gamma[\e\bar \psi_k](2\pi)|_{\e=0}\right)\,.
\]
Here, $\epsilon_{jk}=1$ if $j<k$, and $\epsilon_{jk}=-1$ if $j>k$.
Using \eqref{eq:36}, we get
\begin{equation}
\begin{split}
M_{3i}= & \epsilon_{jk}\iint_0^{2\pi}\d s\d t \,\bar\psi_j(s)\bar\psi_k(t)\,
  \bar\gamma(2\pi)\cdot\left(\bar\gamma(2\pi)\wedge \bar\gamma(s)\right)\wedge
  \left(\bar\gamma(2\pi)\wedge \bar\gamma(t)\right)\\
  = & \epsilon_{jk}\iint_0^{2\pi}\d s\d t \,\bar\psi_j(s)\bar\psi_k(t)
  \left(\bar\gamma(2\pi)\cdot (\bar\gamma(s)\wedge \bar\gamma(t))\right)\\
  = & \epsilon_{jk}h\iint_0^{2\pi}\d s\d t \,\bar\psi_j(s)\bar\psi_k(t) b(s,t)+O(h^2)\,,
\end{split}
\label{eq:17}
\end{equation}
where
\[
b(s,t)=\left(\sin (t-s)
    w(2\pi)-\sin(t)w(s)+\sin(s)w(t)\right)\,.
\]
Next we compute the partial derivatives of the third component of $f_h$, 
\begin{equation}
\begin{split}
\left.\frac{\partial(f_h)_3}{\partial a_i}\right|_{a=0}=&   
\left(\ddt{\bar\gamma}(2\pi)\wedge \bar\gamma(2\pi)\right)\cdot
\frac{\d}{\d\e}\left.\left(\ddt{\bar\gamma}[\e\bar\psi_i](2\pi)\right)\right|_{\e=0} \\
=& \left(\ddt{\bar\gamma}(2\pi)\wedge
  \bar\gamma(2\pi)\right)\cdot \int_0^{2\pi} \d s \bar\psi_i(s)
  \ddt{\bar\gamma}(2\pi)\wedge\bar\gamma(s) \\
=& \int_0^{2\pi} \d s \bar\psi_i(s) \bar\gamma(s)\cdot
  \bar\gamma(2\pi) \\
=& \cos(s)+O(h)\,,
\end{split}\label{eq:18}
\end{equation}
where we have used \eqref{eq:36} in the second equation.
\[
\begin{split}
\det D f_h=&h\int \left(\prod_{i=1}^3 \bar\psi_i(s_i)\d
  s_i\right)\cos(s_1)b(s_2,s_3)-\cos(s_2)b(s_1,s_3))+\cos(s_3)b(s_1,s_2)+O(h^2)\\
= & h \int \left(\prod_{i=1}^3 \bar\psi_i(s_i)\d
  s_i\right)\tilde b(s_1,s_2,s_3)+O(h^2)\,,
\end{split}
\]
where 
\[
\tilde
b(s_1,s_2,s_3)=\left(\sin(s_2-s_1)w(s_3)+\sin(s_1-s_3)w(s_2)+\sin(s_3-s_2)w(s_1)\right)\,.
\]
The only solutions of $\bar b(s_1,s_2,s_3)=0$ for $w$ are $w(s)=A\cos (s)$,
$A\in \R$. Since this is not possible by $(u,v,w)\in\A$,
$\bar\psi$ can be chosen such that 
\begin{equation}
  \label{eq:19}
  \det Df_h\geq C h\,.
\end{equation}
 \\
Next, we estimate
$|Df_h^{-1}(0)|$ using the formula
\begin{equation}
Df_h^{-1}=\frac{1}{\det Df_h}\mathrm{cof}\,Df_h\label{eq:matrixinvert}\,,
\end{equation}
where 
\[
(\mathrm{cof}\,Df)_{ij}=(-1)^{i+j}M_{ij}
\]
is the cofactor matrix of $Df$. Since by \eqref{eq:36},
$|Df_h(0)|\leq C$, where $C$ is independent of $h$, the same holds true for
$\mathrm{cof}\,Df_h$. Hence, by \eqref{eq:19} and \eqref{eq:matrixinvert}, we have shown
\eqref{eq:3}. The property
\eqref{eq:4} follows from
\begin{equation}
\begin{split}
  \bar\gamma(2\pi)-\bar\gamma(0)=&O(h^4)\\
  \ddt{\bar\gamma}(2\pi)-\ddt{\bar\gamma}(0)=&O(h^4)\,,
\end{split}\label{eq:O4diff}
\end{equation}
which holds by Lemma \ref{lem:defo14}.
\end{proof}

The approximation of $(u,v,w)\in\A$ by $W^{2,\infty}$ functions is the content
of the following lemma.
\begin{lemma}
\label{lem:wapprox}
Let $w\in W^{2,2}(S^1)$ with $w\geq 1$, $\int_{S^1}w^2-w'^2\d t=0$. There exists
a sequence $w_\e\in W^{2,\infty}(S^1)$ such that
\[
\begin{split}
   w_\e\geq &1\\
  \int_{S^1} w_\e^2- w_\e'^2\d t=&0\\
  w_\e\to &w\quad\text{in } W^{2,2}(S^1)\text{ as }\e\to 0\,.
\end{split}
\]
\end{lemma}
\begin{proof}
Let $\eta$ be a standard mollifier, i.e.,
$\eta\in C^\infty_0((-1,1))$,  $\eta\geq
0$,  $\eta(t)=0$ for $|t|>1$, $\int \eta(t)\d t=1$. 
Moreover let
$\eta_\e(\cdot)=\e^{-1}\eta(\cdot/\e)$ and $\bar w_\e=\eta_\e*w$. By $w\geq 1$, $\int\eta_\e=1$ and
$\eta_\e\geq 0$, we have $\bar w_\e\geq 1$. By standard properties of the
convolution with  mollifiers,
\[
\begin{split}
  \eta_\e * w\to & w \quad\text{ in }L^2(S^1)\text{ as }\e\to 0\,,\\
  \eta_\e * w'\to & w' \quad\text{ in }L^2(S^1)\text{ as }\e\to 0\,.
\end{split}
\]
In particular,
\begin{equation}
\int\bar w_\e^2-\bar w_\e'^2\d t\to 0\,.\label{eq:7}
\end{equation}
For $\psi\in C^\infty(S^1)$, let 
\[
\begin{split}
  G_\psi:\R\to& \R\\
  \lambda\mapsto &\int_{S^1}(w+\lambda\psi)^2-(w'+\lambda\psi')^2\,.
\end{split}
\]
The derivative of $G_\psi$ at 0 is given by
\[
\begin{split}
  DG_\psi(0)=&\frac{\partial}{\partial\lambda}\left.\int_{S^1}(w+\lambda\psi)^2-(w'+\lambda\psi')^2\right|_{\lambda=0}\\
  =& 2\int_{S^1}(w+w'')\psi\,.
\end{split}
\]
We claim that it is possible to choose $\psi$ with $\supp\psi\subset
U:=\{t:w(t)>1\}$ such that $DG_\psi(0)\neq 0$. To see this, note first that by the
continuity of $w$, $U$ is open. By $w\geq 1$ and
 $\int_{S^1}w^2-w'^2=0$, $U$ is non-empty. Assuming $\int_U(w''+w)\psi=0$ for all
 $\psi\in C^\infty_0(U)$, we  have
\begin{equation}
\|w''+w\|_{L^2(U)}=0\,.\label{eq:8}
\end{equation}
Let $U_0$ be a connected component of $U$. By \eqref{eq:8}, $w(t)=A\sin
(t+\alpha)$ for $t\in U_0$ for some $A\in \R,\alpha\in S^1$. Let
$t_0\in \partial U_0$. Then $1=w(t_0)=\lim_{t\to t_0}A \sin (t+\alpha)$, and
$t_0$ is not a local maximum of the latter function. By the embedding
$W^{2,2}(S^1)\subset C^{1}(S^1)$,  $w'$ is continuous. Hence
we have $w'(t_0)=\lim_{t\to t_0} A\cos (t+\alpha)\neq 0$. On the other hand,
again by the continuity of $w'$, we must have $w'(t_0)=0$ (since there is no $t$
in any neighborhood of $t_0$ with $w(t)<w(t_0)$). This contradiction
proves that it is possible to choose $\psi\in C_0^\infty(U)$ such that
$DG_\psi(0)=2\int_{S^1}(w+w'')\psi\neq 0$.\\
Choose such a $\psi$, and let $\delta_1$ be such that $w\geq 1+2\delta_1$ on
$\supp \psi$.
For $\e$ small enough,
we have $\supp \psi\subset \{t:\bar w_\e(t)\geq 1+\delta_1\} $.
Again by standard properties of approximation by mollifiers, we have
\[
\int_{S^1}(\bar w_\e''+\bar w_\e)\psi\to \int_{S^1}( w''+ w)\psi
\]
as $\e\to 0$. Hence there exists $\delta_2>0$ such that
\[
|DG^\e_\psi(0)|\geq \delta_2 \quad\text{uniformly in }\e\,,
\]
where $G^\e_\psi$ is defined by $\lambda\mapsto \int_{S^1}(\bar
w_\e+\lambda\psi)^2-(\bar w_\e'+\lambda\psi')^2$. \\
Now we are going to apply the implicit function theorem,
Theorem~\ref{thm:implicit}, with $f=G^\e_\psi$, $x_0=0$. Condition a) from that
theorem can be fulfilled with $\alpha$ arbitrarily small, if we choose $\e$
small enough. Condition b) is easily verified by direct computation,
\[
|DG_\psi^\e(\lambda_1)-DG_\psi^\e(\lambda_2)|=2|\lambda_1-\lambda_2|\left|\int_{S^1}(\psi''+\psi)\psi\right|\,.
\]
Finally, property c) holds since $\alpha$ can be chosen arbitrarily
small. Hence, we get the existence of $\lambda_\e$ such that
\[
G_\psi^\e(\lambda_\e)=0\,,
\]
with $\lambda_\e\to 0$ as $\e\to 0$.
Thus, again by choosing $\e$ small enough, we get
\[
w_\e:=\bar w_\e+\lambda_\e\psi\geq 1 \quad \text{ on }S^1\,.
\]
The sequence $w_\e$ fulfills all the required properties. This proves the lemma.
\end{proof}

\begin{corollary}
\label{cor:Aapprox}
Let $(u,v,w)\in \A$. Then there exists a sequence $(u_\e,v_\e,w_\e)\in\A\cap
W^{2,\infty}(S^1;\R^3)$ with $(u_\e,v_\e,w_\e)\to (u,v,w)$ in
$W^{2,2}(S^1;\R^3)$.
\end{corollary}
\begin{proof}
Let $w_\e$ be the approximation of $w$ from Lemma \ref{lem:wapprox}. We set
\[
\begin{split}
  u_\e:=&-w_\e^2/2\\
  v_\e(t):=&v(0)-\int_0^tu_\e(t)+\frac{w_\e'(t)^2}{2}\d t\,.
\end{split}
\]
This sequence has all required properties.
\end{proof}

\begin{proof}[Proof of Theorem \ref{thm:gamma}, upper bound]
By Corollary  \ref{cor:Aapprox} and a standard diagonal sequence argument, it
suffices to construct the recovery sequence for the case $(u,v,w)\in \A\cap
W^{2,\infty}(S^1;\R^3)$.\\
Let $\bar\gamma=\gamma_h^{(4)}$ be as in the conclusion of Lemma 
\ref{lem:defo14}. By Lemma \ref{lem:assver} and Theorem \ref{thm:implicit}, there exists $\bar\psi\in
C_0^\infty(0,2\pi)$ (independent of $h$) and $a_h\in B_{C h^3}(0)\subset \R^4$
such that $\bar\gamma[a_h\cdot\bar \psi](2\pi)=\bar\gamma[a_h\cdot\bar\psi](0)$
and $\ddt{\bar\gamma}[a_h\cdot\bar \psi](2\pi)=\ddt{\bar\gamma}[a_h\cdot\bar\psi](0)$.
Set $\gamma_h=\bar\gamma[a_h\cdot\bar \psi]$. By the boundary values  of
$\gamma_h,\gamma_h'$ at $0$ and $2\pi$, we may view $\gamma_h$ as a function in $W^{2,2}(S^1;\R^3)$. By $a_h=O(h^3)$ and
$\bar\gamma\cdot e_z\geq h+\frac12 h^{5/2}$, we have $\gamma_h\cdot e_z\geq h$ for $h$ small
enough. Thus, $\gamma_h$ fulfills the constraints $|\gamma_h|=|\gamma_h'|=1$,
and $\gamma_h\cdot e_z\geq h$. Finally, $\E_h(\gamma_h)\to \E^0(u,v,w)$ follows from the convergence
$h^{-1}(u_h,v_h,w_h)\to (0,0,w)$ and \eqref{eq:1}. Hence $\gamma_h$ is the desired recovery sequence.
\end{proof}

\section{Minimizers of the limit functional}
\label{sec:minim-limit-funct}
To analyze the minimizers of the limiting functional $\E^0$, we introduce
\[
\begin{split}
  \bar\E^0:W^{2,2}(S^1)\to&\R\\
  w\mapsto&\begin{cases}\int_0^{2\pi}(w''+w)^2\d t & \text{ if
    }\int_0^{2\pi}(w^2-w'^2)\d t=0\\
    +\infty & \text{ else.}
  \end{cases}
\end{split}
\]
It is easily seen that $\E^0(u,v,w)<\infty$ only if $\bar\E^0(w)<\infty$, and in
that case $u$ and $v$ are (up to a constant) uniquely determined by $w$. Thus
the study of minimizers of $\E^0$ reduces to the study of minimizers of
$\bar \E^0$.\\
\\
The existence of minimizers of $\bar\E^0$ follows in an obvious way by an
application of the direct method.
It is possible to compute rather explicitly the minimizers, provided one knows
that they are in $C^2(S^1)$. The proof of this fact is one the main points of
the following theorem. It will turn out that the following functions $(0,\pi]\to
\R\cup\{\pm\infty\}$ play an important role 
for the characterization of $w$:
\[
\begin{split}
  g_\alpha(z):=&-\frac{\alpha^2\sin(z)\cos(\alpha z)-\alpha\sin(\alpha
    z)\cos(z)}{\sin(z)\cos(\alpha z)-\alpha\sin(\alpha z)\cos(z)}\\
  \tilde g_\alpha(z):=&\frac{\alpha^2\sin(z)\cosh(\alpha z)-\alpha\sinh(\alpha
    z)\cos(z)}{\sin(z)\cosh(\alpha z)+\alpha\sinh(\alpha z)\cos(z)}\,.
\end{split}
\]
For  plots of $g_\alpha$, $\tilde g_\alpha$ for $\alpha=7$, see
Figures \ref{fig:1} and \ref{fig:2}.
\begin{figure}[h]
\includegraphics[height=5cm]{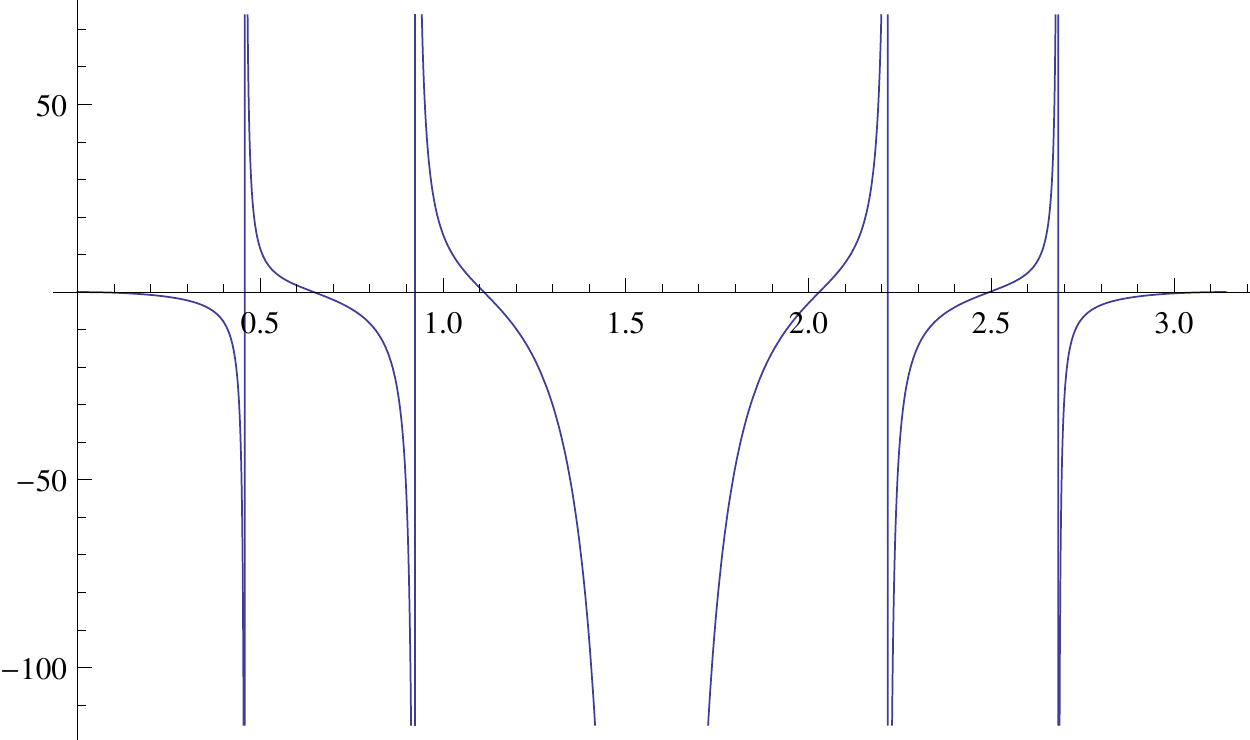}
\caption{Plot of $g_7$. }
\label{fig:1}
\end{figure}
\begin{figure}[h]
\includegraphics[height=5cm]{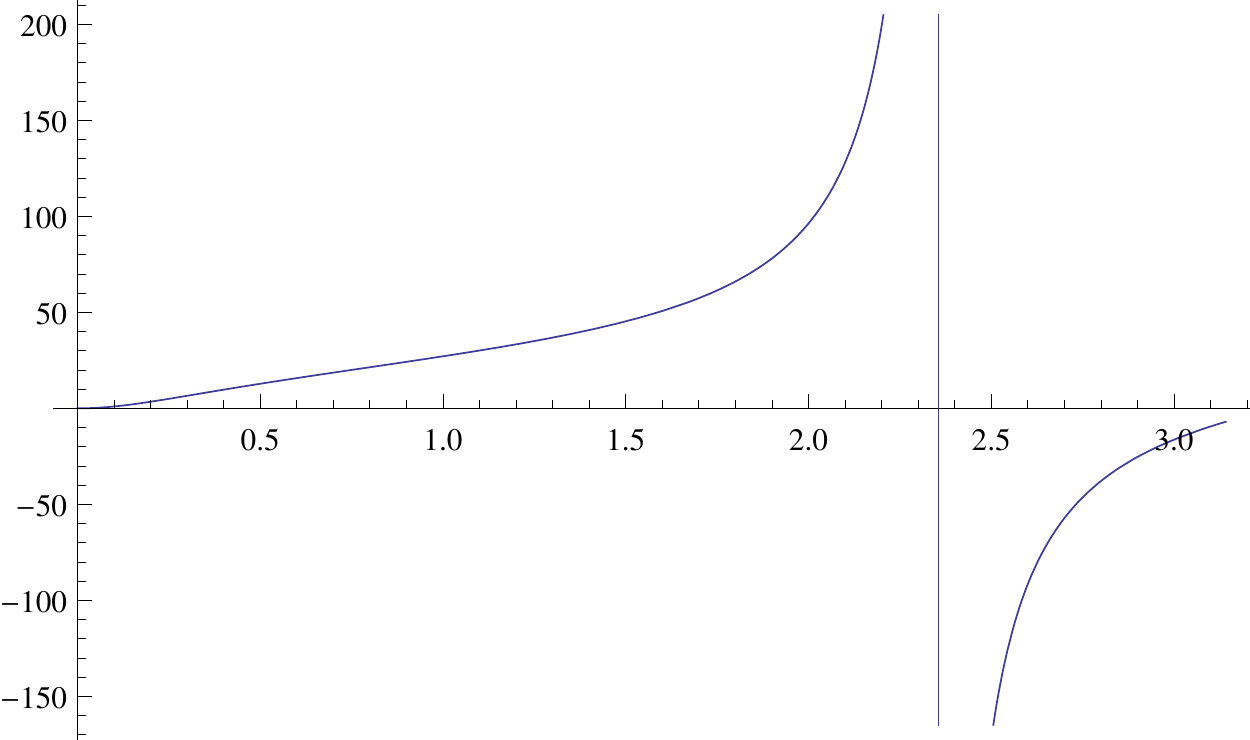}
\caption{Plot of $\tilde g_7$. }
\label{fig:2}
\end{figure}
\begin{rem}
\label{rem:gak}
If $\alpha\in(0,\infty)\setminus\{1\}$, then $g_\alpha^{-1}(k)$ is finite for every $k\in
[0,\infty)$. This follows easily from the fact that $g_\alpha$ is a quotient of
linearly independent
trigonometric polynomials. Analogously, if $\alpha>0$, then $\tilde g_\alpha^{-1}(k)$ is finite for every
$k\in [0,\infty)$. 
\end{rem}
\begin{theorem}
\label{thm:lincons}
Let $w$ be a local minimizer of the functional $\bar\E^0$.
Then there exist $\lambda\in\R\setminus\{0,-1\}$, $k\in [0,\infty)$ and finitely
many mutually disjoint open intervals $I_i\subset S^1$, $i=1,\dots, m$ such that
\begin{itemize}
\item[(i)] $w>1$ on $I:=\cup_iI_i$, $w=1$ on $S^1\setminus I$, and $w\in C^{2,1}(S^1)$.
\item[(ii)] $w''=k$ on $S^1\setminus I$
\item[(iii)] For $i=1,\dots, m$, writing $I_i=(t_i-z_i,t_i+z_i)$ and $\alpha=\sqrt{|1+\lambda|}$, we have:
\begin{itemize}
\item[(a)] If $\lambda\geq -1$, then $z_i\in g_\alpha^{-1}(k)$
and
\begin{equation}
\begin{split}
    w(t)=& \frac{\sin(z_i)\cos(\alpha (t-t_i))-\alpha\sin(\alpha
      z_i)\cos(t-t_i)}{\sin(z_i)\cos(\alpha z_i)-\alpha\sin(\alpha
      z_i)\cos(z_i)}\quad\text{ for }t\in I_i
  \end{split}
\end{equation}
\item[(b)] If $\lambda < -1$, then $z_i\in \tilde g_\alpha^{-1}(k)$ and
\begin{equation}
 \begin{split}
       w(t)=\frac{\sin(z_i)\cosh(\alpha (t-t_i))+\alpha\sinh(\alpha
      z_i)\cos(t-t_i)}{\sin(z_i)\cosh(\alpha z_i)+\alpha\sinh(\alpha
      z_i)\cos(z_i)}\quad\text{ for }t\in I_i
  \end{split}
\end{equation}
\end{itemize}
\end{itemize}


\end{theorem}
The proof of the theorem relies on (a special case of) Theorem 3.1 from \cite{dmitruk1980lyusternik}, which we cite now.
\begin{theorem}
Let $X$ be a normed space, 
$K\subset X$ a closed convex cone with $0\in K$ such that $K-K=X$, $f,g:X\to \R$
Fr\'echet differentiable, $u_0\in X$ 
a local minimum of the variational problem 
\[\left\{
\begin{split}
  f(u)\rightarrow&\min \\
g(u)=& \,0\\
u\in& \,x_0+K\,.
\end{split}\right.
\]
Then the following Lagrange multiplier rule holds: there exists $\lambda\in\R$
such that for every $ \varphi\in K$,
\[
Df(u_0)\varphi+\lambda Dg(u_0)\varphi\geq 0\,.
\]
\label{thm:convar}
\end{theorem}
\begin{proof}[Proof of Theorem \ref{thm:lincons}]
Denote by $I_i$ the connected components of $I:=\{x\in S^1:w(x)>1\}$. There are 
at most countably many of them, and we may write $I=\cup_{i=1}^\infty I_i$. We
are going to prove the statements (i), (ii) and (iii) for every $i=1,\dots,\infty$, and conclude in the end that there
is only a finite number of the $I_i$'s.\\
We apply  Theorem \ref{thm:convar} with $X=W^{2,2}(S^1)$, 
\[
\begin{split}
  f(\bar w)=&\int_{S^1} \left(\bar w''(t)+\bar w(t)\right)^2\d t\\
g(\bar w)=&\int_{S^1} \left(\bar w^2(t)-\bar w'^2(t)\right)\d t,
\end{split}
\]
taking $u_0=w$ as the minimizer from the statement of the present theorem,  and
\[
K=\{\varphi\in W^{2,2}(S^1):\varphi(t)\geq 0 \text{ for }t\in S^1\setminus I\}\,.
\] 
The conditions from Theorem \ref{thm:convar} are fulfilled and hence we obtain
the existence of some $\lambda\in \R$ such that
\begin{equation}
\int_{S^1} \left(\ddtt{4}{w}+(2+\lambda)w''+(1+\lambda)w \right)\varphi \d t\geq 0\,,\label{eq:29}
\end{equation}
for all $\varphi\in K$, where the integral is defined by integration by parts. In particular, this last inequality holds true for every
$\varphi\in W^{2,2}(S^1)$ with $\varphi\geq 0$. By the Riesz-Schwartz Theorem
(see e.g.\cite{MR0035918}), the linear map
\[
\begin{split}
  \bar \mu: W^{2,2}(S^1)\to&\R\\
  \varphi\mapsto&\int_{S^1}\left(\ddtt{4}{w}+(2+\lambda)w''+(1+\lambda)w\right)\varphi\d
  t
\end{split}
\] 
defines a non-negative Radon 
measure $\mu$ on $S^1$. In
particular, it follows that $\ddtt{3}{w}\in L^\infty(S^1)$ and hence 
$w\in C^{2,1}(S^1)$. This proves (i).\\
\\
For any $\varphi\in W^{2,2}(S^1)$ with $\supp\,\varphi\subset I$, we have
$\varphi\in K$, $-\varphi\in K$, and hence by \eqref{eq:29}, 
\begin{equation*}
\int_{S^1} \left(\ddtt{4}{w}+(2+\lambda)w''+(1+\lambda)w \right)\varphi \d t= 0\,.
\end{equation*}
This implies, in the sense of Radon measures,
\[
\ddtt{4}{w}+(2+\lambda)w''+(1+\lambda)w=0 \quad \text{ on }I\,.
\]
We claim that 
\begin{equation}
S^1\neq I\,.\label{eq:33}
\end{equation}
 Indeed, assume  the contrary  were the case. Then $w$ is a
local minimizer of the variational problem
\[
\left\{\begin{split}
  \int_{S^1}(w''+w)^2\d t\to& \text{min}\\
  \int_{S^1}(w^2-w'^2)\d t=&0\,.
\end{split}\right.
\]
By the standard Lagrange multiplier formalism, there exists $\tilde\lambda\in \R$ such
that
\begin{equation}
\ddtt{4}{w}+(2+\tilde\lambda)w''+(1+\tilde\lambda)w=0\quad\text{ on } S^1\,.\label{eq:32}
\end{equation}
Identifying $S^1$ with the interval $(-\pi,\pi)$, we must have
\begin{equation}
\begin{array}{rlrl}
w(-\pi)=&w(\pi)\qquad &w'(-\pi)=&w'(\pi)\\
w''(-\pi)=&w''(\pi)\qquad &\ddtt{3}{w}(-\pi)=&\ddtt{3}{w}(\pi)
\end{array}\label{eq:34}
\end{equation}
We claim that the dimension of the the solution space of the boundary value problem
defined by \eqref{eq:32} and \eqref{eq:34} is zero for $\tilde\lambda\neq -1$.
Indeed, rewriting \eqref{eq:32} as 
$x'=Ax$
with
\[
x=\left(\ddtt{3}{w}(t),w''(t),w'(t),w(t)\right)^T,\quad
A=\left(\begin{array}{cccc}0&2+\tilde\lambda&0&1+\tilde\lambda\\1&0&0&0\\0&1&0&0\\0&0&1&0\end{array}\right)
\]
and \eqref{eq:34} as $Ux=Mx(-\pi)+Nx(\pi)=0$
with $M=-N= \id_{4\times 4}$, we get that the ``boundary form'' $U$ applied to
the fundamental matrix $\exp(A\cdot)$ is given by
$U\exp(A\cdot)=\exp(-A\pi)-\exp(A\pi)$, which has full rank unless $\tilde\lambda=-1$.
Hence, for $\tilde \lambda\neq -1$, 
the claim that the dimension of the solution space is zero 
follows from a standard result in ODE theory (see e.g.~\cite[Chapter 11, Theorem
3.3]{MR0069338}). Hence, $w(t)=0$ is the unique solution to the boundary value problem above,
which is a contradiction to $w\geq 1$. \\
If $\tilde\lambda=-1$, then the solutions to \eqref{eq:32} are given
by  $w''(t)=-a \cos (t+t_0)$, where $a,t_0\in\R$ are integration constants. This
implies $w(t)=a\cos(t+ t_0)+b$, where $b\in \R$ is yet another integration
constant. From the constraint $\int_{-\pi}^\pi (w^2-w'^2)=0$, it follows $b=0$,
which again produces a contradiction to $w\geq 1$.
 This proves \eqref{eq:33}.\\
Now fix some $I_i$. After a translation, we may write $I_i=(-z_i,z_i)$ for some
$z_i\in(0,\pi]$ by \eqref{eq:33}. 
By the regularity of $w$, we have $w(\pm z_i)=1$, $w'(\pm z_i)=0$. 
Hence, $w|_{\overline{I_i}}$ has to be a solution of the boundary value problem
\begin{equation}
\left\{ \begin{array}{rl}\ddtt{4}{w}+(2+\lambda)w''+(1+\lambda)w=&0 \quad\text{on
   }(-z_i,z_i)\\w(-z_i)=w(z_i)=&1\\
w'(-z_i)=w'(z_i)=&0\,.
\end{array}\right.\label{eq:31}
\end{equation}
If $\lambda=0$,  there do not  exist  any solutions to \eqref{eq:32}. This would
imply $I=\emptyset$, which cannot be by the constraint $\int w^2-w'^2=0$. Hence,
we conclude 
\begin{equation}
\lambda\neq 0\,.\label{eq:35}
\end{equation}
If $\lambda\neq 0$, there exists a unique solution to \eqref{eq:31}.
We recall the notation $\alpha=\sqrt{|1+\lambda|}$. The solution of \eqref{eq:31} is given by
\begin{equation}
w(t)=
\frac{\sin(z_i)\cos(\alpha t)-\alpha\sin(\alpha z_i)\cos(t)}{\sin(z_i)\cos(\alpha z_i)-\alpha\sin(\alpha z_i)\cos(z_i)}\quad\text{
  if }\lambda\geq -1\,,
\label{eq:30}
\end{equation}
and
\begin{equation}
  \label{eq:40}
w(t)=  \frac{\sin(z_i)\cosh(\alpha t)+\alpha\sinh(\alpha z_i)\cos(t)}{\sin(z_i)\cosh(\alpha z_i)+\alpha\sinh(\alpha z_i)\cos(z_i)}\quad\text{
  if }\lambda < -1\,.
\end{equation}
Next, 
we prove (ii).
By the explicit formulas
\eqref{eq:30}, \eqref{eq:40}, we see that $w''$ is constant on the boundary of
every  $I_i$ (which of course just consists of up to two points). Let $k_i$
denote the value of $w''$ on $\partial I_i$. Then we set 
\[
v(t)=\begin{cases}k_i&\text{ if } t\in I_i\\w''(t)&\text{ if } t\in S^1\setminus
  I\,.\end{cases}
\]
By the regularity of $w$, $v$ is Lipschitz. Further, $v'=0$ on $I$ and on
$S^1\setminus  I=\{w=1\}$, we have $v'=\ddtt{3}{w}=0$ almost everywhere. Hence
$v'=0$ almost everywhere in $S^1$ and
$v$ is constant. This proves (ii).\\
Let us consider some 
$I_i$. Again, after translation, we may write $I_i=(-z_i,z_i)$. Note
that $g_\alpha(z_i)=w''(z_i)$ if $\lambda\geq -1$ and $\tilde
g_\alpha(z_i)=w''(z_i)$ if $\lambda<-1$. 
By (ii) and the continuity of $w''$, we have
$z_i\in g_\alpha^{-1}(k)$ if  $\lambda\geq -1$ and $z_i\in\tilde
g_\alpha^{-1}(k)$ if $\lambda<-1$. This holds true for any $i$. In combination
with \eqref{eq:30} and \eqref{eq:40}, this shows (iii).\\
\\
It remains to show that there is only a finite number of connected components of
$I$. First assume $\lambda\geq -1$. We may restate the relation $z_i\in g_\alpha^{-1}(k)$ as
\[
\frac12 \L^1(I_i)\in g_\alpha^{-1}(k) \quad \text{ for all } i\,,
\]
where $\L^1$ denotes the one-dimensional Lebesgue measure.
We claim that $\alpha\not\in\{0,1\}$. We have already seen $\lambda\neq 0$ in
\eqref{eq:35}, and
hence $\alpha\neq 1$. Further, if $\alpha=0$, then $w=1$
everywhere  by \eqref{eq:30}. This is a contradiction to $\int w^2-w'^2=0$. 
Hence, as we have noted in Remark \ref{rem:gak} above,
$g_\alpha^{-1}(k)$ is a finite set. In particular, there is a certain minimal length that any $I_i$ can have. This
implies that there are only finitely many connected components of $I$:
\[
I=\bigcup_{i=1}^m I_i\,.
\] 
If $\lambda<-1$, one argues in exactly the same way (using the finiteness of
$\tilde g_\alpha^{-1}(k)$) to conclude that there are
finitely many connected components of $I$.
\end{proof}
\bibliographystyle{plain}
\bibliography{regular}

\end{document}